\documentclass[letterpaper, 10 pt, conference]{ieeeconf}
\IEEEoverridecommandlockouts
\usepackage{cite}
\usepackage{graphicx}
\usepackage{color}
\usepackage{textcomp}
\usepackage{amsmath}
\usepackage{amsfonts}
\usepackage{amssymb}
\usepackage{mathtools}
\usepackage{xcolor}
\usepackage{bm}
\usepackage{float}
\usepackage{comment}
\usepackage{subcaption}
\usepackage[font={small}]{caption}
\graphicspath{ {images/} }
\newtheorem{proposition}{\textbf{Proposition}}

\usepackage{algorithm} 
\usepackage{algpseudocode} 

\title{Prediction of fitness in bacteria with causal jump dynamic mode decomposition}

\author{Shara Balakrishnan, Aqib Hasnain, Nibodh Boddupalli, Dennis M. Joshy, Robert G. Egbert, and Enoch Yeung %
\thanks{Shara Balakrishnan ({\tt\small sbalakrishnan@ucsb.edu}) is with the Department of Electrical and Computer Engineering, University of California, Santa Barbara \newline \indent Aqib Hasnain, Nibodh Boddupalli and Dennis Joshy are with the Department of Mechanical Engineering, University of California, Santa Barbara, Robert G. Egbert is with the Environmental and Biological Sciences Directorate at the Pacific Northwest National Laboratory, Richland, WA.  \newline \indent Enoch Yeung ({\tt\small eyeung@ucsb.edu}) is with the Department of Mechanical Engineering, Center for Control, Dynamical Systems, and Computation, and Biomolecular Science and Engineering, University of California, Santa Barbara}
}

\begin{document}

\maketitle


\begin{abstract}
In this paper, we consider the problem of learning a predictive model for population cell growth dynamics as a function of the media conditions. We first introduce a generic data-driven framework for training operator-theoretic models to predict cell growth rate. We then introduce the experimental design and data generated in this study, namely growth curves of {\it Pseudomonas putida} as a function of casein and glucose concentrations. We use a data driven approach for model identification, specifically the nonlinear autoregressive (NAR) model to represent the dynamics. We show theoretically that Hankel DMD can be used to obtain a solution of the NAR model.  We show that it identifies a constrained NAR model and to obtain a more general solution, we define a causal state space system using 1-step, 2-step,..., $\tau$-step predictors  of the NAR model and identify a Koopman operator for this model using extended dynamic mode decomposition.  The hybrid scheme we call causal-jump dynamic mode decomposition, which we illustrate on a growth profile or fitness prediction challenge as a function of different input growth conditions. We show that our model is able to recapitulate training growth curve data with 96.6\% accuracy and predict test growth curve data with 91\% accuracy.


\end{abstract}


\section{Introduction}\label{intro}
One of the most fundamental processes in life is the ability to replicate and pass on hereditary material \cite{kornberg1980replication}. From viral particles to bacteria to mammalian cells, cell division is fundamental to growth, maintenance of physiological health, and intrinsically tied to the notion of senescence \cite{mathon2001cell}.  

The mechanisms for controlling growth in organisms are determined by metabolic networks \cite{wu2016metabolic,glazier2015metabolic}, namely their topological structure and parametric realization.  Known metabolic networks in well studied model organisms such as {\it E. coli} \cite{de2016growth} and {\it S. cerevisiae } \cite{sanchez2017improving,zwietering1990modeling} have given rise to predictive models that translate environmental activity to metabolic network state, and ultimately to predictions of growth rate.  For canonical biological model systems, these models have been highly accurate in predicting growth rate and found utility in industrial microbiology applications, e.g. in the design of bioreactors or informing best practices in food safety. 

For many biological life forms, relatively little is known about their metabolic network or structure. This is especially the case when  developing bioengineering tools in novel host microbes \cite{tschirhart2019synthetic,Khan2018broad}. For new organisms, canonical metabolic networks are lacking and often obtained through a process of sequence alignment and comparative analysis with existing metabolic network models in relative species. However, many novel strains do not exhibit significant similarity, and even in the case of sequence similarity, small mutations can lead to dramatically different growth phenotypes, e.g. growth of non-pathogenic soil strains\cite{gill1979effect,gulliver2016comparative} versus pathogenic counterparts \cite{LaBauve2012growth}.  The absence of predictive cross-species models, as well as the inability to predict growth phenotype wholly from sequence data, motivates the need for data-driven methods to accelerate the discovery of metabolic models and growth rate prediction models.  

Due to advances in high-throughput experimental techniques, it is relatively easy to characterize growth rates as a function of exposure to environment.  Liquid  and acoustic-liquid handling robotics enables interrogation of thousands of growth conditions in a single microtiter plate, which in turn opens the door for using data-driven approaches \cite{Palacios2014bayesian} to predict growth rate as a function of environmental state.  Is it possible to accurately predict the growth rate of a microbe, entirely from the chemical composition and environmental parameters of its growth condition?  In this paper we explore a data-driven operator theoretic approach that utilizes microtiter plate reader data, and more generally multi-variate time-series data, to develop predictive models of growth rate in {\it Pseudomonas putida}, a broadly used strain for commercial bioreactors and a target workhorse for tractable genetic engineering \cite{lee2016vibrio}.      

A broadly successful class of data-driven modeling approaches stem from the study of Koopman operators, a mathematical construct for representing the time-evolution of nonlinear dynamical systems.  In Koopman operator theory, the time-evolution of a nonlinear system is defined on a function space, acting on the original state of the system.  In this function space, known the observables space, the Koopman operator is a linear operator, enabling spectral analysis, the decomposition of eigenspaces, and study of nonlinear structure \cite{mezic2005spectral}.   The Koopman operator framework has been developed for continuous \cite{budivsic2012applied} and discrete time systems \cite{williams2015data,rowley2009spectral}, for open-loop \cite{williams2015data} and input-controlled \cite{proctor2016dynamic,williams2016extending} dynamic systems.   Thus, Koopman operators present a powerful framework for analyzing the behavior of nonlinear systems, including predicting how experimental conditions regulate growth dynamics.  

Many numerical methods for identifying Koopman operators from data have been developed in the last two decades \cite{askham2018variable,kaneko2019convolutional,azencot2019consistent,manohar2019optimized,schmid2010dynamic,sinha2019computation,yeung2019learning,hasnain2019data}.  The most common approach is to use dynamic mode decomposition (DMD), which models nonlinear dynamics via an approximate local linear expansion \cite{schmid2010dynamic}.  In \cite{williams2015data} an extended dictionary of basis functions with universal function approximation properties is used to discover an approximation of the lifting map or observables. These techniques suffer from combinatorial explosion, which generally has prohibited analysis of high-dimensional nonlinear systems \cite{Johnson2018class}. The most recent developments in the field of DMD integrate established advances in deep learning with DMD \cite{yeung2019learning,otto2019linearly,takeishi2017learning,li2017extended} where the deep neural networks have the capacity to approximate exponentially many distinct observable functions. Recent work has shown deep Koopman learning algorithms can be extended to synthesize controllers for systems subject to uncertainty \cite{you2018deep}, suggesting that deep Koopman learning can be used broadly for robust controller synthesis.

The existing algorithms assume a full state measurement and construct observables from that. With partial state observables like in biological systems, we construct a state space model for an output nonlinear difference equation model and identify a Koopman operator for that model. In Section \ref{sec:Koop} we briefly introduce the Koopman operator and DMD and the existing literature.  In Section \ref{sec:ExptSetup} we describe the experimental setup for obtaining the growth curve data of \textit{Pseudomonas putida} by adding different concentrations of casein and glucose substrates to the media. In Section \ref{sec: Mathematical Model}, we justify nonlinear autoregressive difference equation model is an appropriate choice for this system. In Section \ref{sec: Hankel DMD}, we formulate the Hankel DMD as a solution of the NAR model and bring out its issues and in Section \ref{sec: NAR-DMD}, we formulate a state space model for the NAR model and use extended dynamic mode decomposition to identify a Koopman operator for the NAR model. In Section \ref{sec: Results}, we show that the algorithm is able to train a predictive Koopman operator, that predicts with 3.4\% on the training data and 9\% on the test data on extended forecasting tasks approximately 500 time steps ahead.

\section{Mathematical Preliminaries} \label{sec:Koop}
Consider a discrete-time autonomous nonlinear dynamical system 
\begin{equation}
\begin{split}
 x_{k+1} &=f(x_k)\
\end{split}
 \label{eq:sys}
\end{equation}
with $f:\mathbb{R}^n\rightarrow \mathbb{R}^n$ is analytic. There exists a Koopman operator \cite{yeung2018koopman} of (\ref{eq:sys}), which acts on a function space $\mathcal{F}$ as $\mathcal{K}$ : $\mathcal{F}$ $\rightarrow$ $\mathcal{F}$. This action can be given by
\begin{equation}
\mathcal{K} \psi(x_k) = \psi \circ f(x_k). 
\label{eq:Koopislinear}
\end{equation}
where the function $\psi: \mathbb{R}^n\rightarrow \mathbb{R}$ is called an \textit{observable} of the system and the set of all observables $\psi \triangleq \{\psi_i\}_{i=1}^{p}, p \leq \infty$ on the system. Here $\mathcal{F}$ is invariant under the action of $\mathcal{K}$. 

The most important property of the Koopman operator that we utilize is the linearity of the operator, in other words,
\begin{equation*}
    \mathcal{K}(\alpha \psi_1 + \beta \psi_2) = \alpha \psi_1 \circ f + \beta \psi_2 \circ f = \alpha \mathcal{K} \psi_1 + \beta \mathcal{K} \psi_2
\end{equation*}
which follows from (\ref{eq:Koopislinear}) since the composition operator is linear. Thus, we have that the Koopman operator of (\ref{eq:sys}) is a linear operator that acts on observable functions $\psi (x_k)$ and propagates them forward in time. 

\subsection{DMD and relevant variants}
The practical identification of Koopman operator for a nonlinear system from input-output data is commonly done using DMD \cite{schmid2010dynamic} or extended DMD \cite{williams2015data} which constructs an approximate Koopman operator $K$. Rowley et. al  showed that the finite-dimensional approximation to the Koopman operator obtained from DMD is closely related to a spectral analysis of the linear but infinite-dimensional Koopman operator \cite{rowley2009spectral}. The approach taken to compute an approximation to the Koopman operator in both DMD and extended DMD is as follows
\begin{equation}
K = \min_{K} || \Psi(X_f) - K\Psi(X_p)|| = \Psi(X_f)\Psi(X_p)^{\dag}
\label{eq:learnKoop}
\end{equation}
where 
$  X_f \equiv \begin{bmatrix} x_{1} & \hdots & x_{N-1} \end{bmatrix},$ $ X_p  \equiv \begin{bmatrix} x_{2} & \hdots & x_{N} \end{bmatrix}
$
are snapshot matrices formed from the discrete-time dynamical system (\ref{eq:sys}),
$
  \Psi(X) \equiv \begin{bmatrix} \psi_1(x) & \hdots & \psi_R(x) \end{bmatrix}
$
is the mapping from physical space into the space of observables and \textsuperscript{$\dag$} denotes the Moore-Penrose pseudoinverse. Here $N$ is the number of snapshots i.e. timepoints. We note that DMD is a special case of extended DMD where $\psi(x) = x$. Throughout the rest of the paper, when we refer to the Koopman operator we mean the finite dimensional approximation to the infinite-dimensional Koopman operator. 

\section{Experimental Setup}\label{sec:ExptSetup}
We describe the procedure adopted to obtain {\it P. putida}'s growth curve for varying concentrations of glucose and casein substrates in the media.

\noindent \textbf{Incubation:} 
We revived {\it P. putida} cryopreserved at $-80^oC$ in 30\% (vol/vol) glycerol stock by suspending a small portion into a polypropylene test tube containing 4 mL of Lysogeny Broth (LB). We cultured it at $30^o C$ spinning with a speed of 200 revolutions per minute (rpm) for 12 hours overnight. A visual inspection of the culture tube resulted in a cloudy culture medium, suggesting subsequent growth with the seed {\it P. putida} culture in a plate reader was feasible.

\noindent \textbf{Solution Preparation:} 
We prepared 300 g/L solution of glucose solution and 225 g/L casein acid hydrolysate solution. Once the bacteria reached a certain optical density(OD), we shifted the culture from LB media to R2A 2x media obtained from Teknova Inc to 2x the required initial OD.   

\noindent \textbf{Serial dilution setup for \textit{P. putida} culture:}
We use a 630 $\mu L$ 96 well plate to create media with different substrate concentrations. Each well of this plate contained 500 $\mu L$ of modified media - 250 $\mu L$ of culture in 2x R2A at 0.4 OD and 120 $\mu L$ containing a mixture of casein and glucose solutions. To vary casein and glucose across the 96 well plate, we perform 2D serial dilution such that the concentration of glucose was halved across columns and concentration of casein is halved across rows as shown in Figure \ref{fig:96well}. Then, the culture was mixed into each well to get a starting OD of 0.2 in 1x R2A media since equal volumes of culture media and substrate solutions were added.

\begin{figure}[h!]
  \centering
  \includegraphics[width=1\columnwidth]{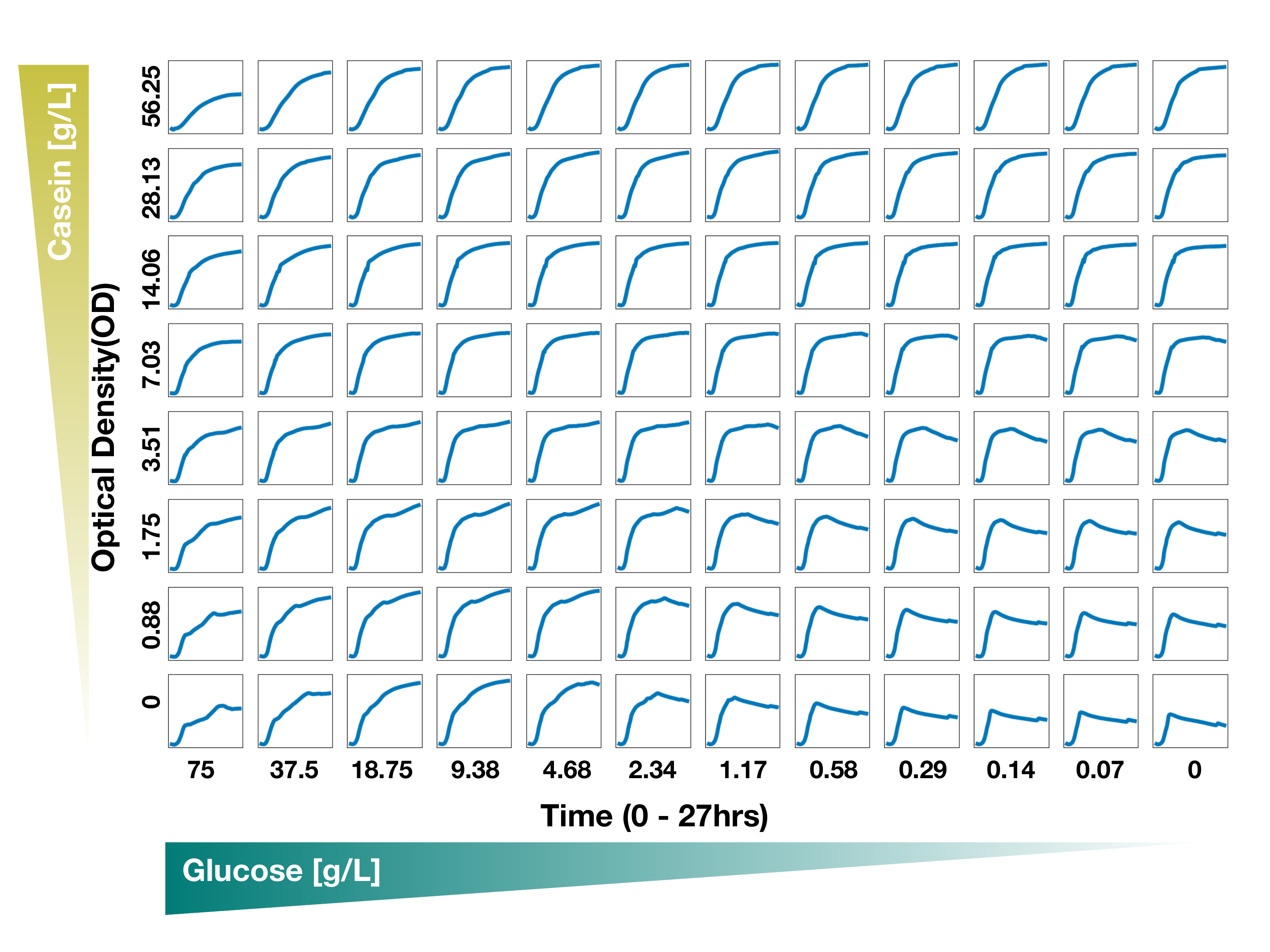}
    \caption{Different initial conditions of substrates obtained by two dimensional serial dilution of casein and glucose and the corresponding growth curves are obtained for a period of 27 hours.}
    \label{fig:96well}
\end{figure}

\textbf{Data Collection:}
The microplate reader was set to $30^o$C and the shaker to 807 cycles per minute, with continuous double orbital mixing. The absorbance at 600 nanometers(nm), which is termed as the Optical Density at 600 nm ($OD_{600}$), was measured as a function of time for 27 hours. We assume in this work, as is widely accepted, that $OD_{600}$ measurements were collected in a linear regime, where cell population is proportional OD600 measurements. The obtained data along with the varying substrate concentrations are shown in Figure. \ref{fig:96well}.

\section{Growth Curve Dynamics Model}\label{sec: Mathematical Model}
The dynamics of the bacterial cell growth can be represented by  
\begin{align}\label{eq: general nonlinear growth curve model}
    \begin{bmatrix}
        N^{(b)}_{k+1}\\
        C_{k+1}\\
        G_{k+1}\ 
    \end{bmatrix}
    =f(N^{(b)}_{k},C_k,G_k)
\end{align}
where the bacterial cell count ($N^{(b)}$), casein substrate concentration ($C$) and glucose substrate concentration ($G$) are the states of the system, $f$ is the nonlinear dynamics and $k$ is the discrete time index. We measure the $OD_{600}$ data as mentioned in section \ref{sec:ExptSetup} and the output equation is given by
\begin{align}\label{eq: output OD600}
    y_k=h(N_k^{(b)})
\end{align}
as $OD_{600}$ is a function of only the number of cells.
Some of the existing empirical nonlinear models for growth curve dynamics include the Monod's model \cite{monod1949growth} which uses a single substrate to form the foundation of the growth curve dynamics and in \cite{brandt2003general} and \cite{kompala1986investigation} multiple substrates are incorporated. Monod's model is a two-state nonlinear dynamical system comprising of the substrate (S) and the number of bacteria ($N^{(b)}$):
\begin{align}
    \dot{N}^{(b)}(t)&=r_{max}\frac{S(t)N^{(b)}(t)}{K_s + S(t)}\nonumber \\
    \dot{S}&=-\gamma\dot{N}^{(b)}\
\end{align}
where $r_{max}$ is the maximum growth rate and $K_s$ is the half velocity constant. As $N^{(b)}$ is the only variable of measurement in (\ref{eq: output OD600}), we convert the model to a single differential equation containing only $N^{(b)}$ 
\begin{equation}\label{eq: Monod's model Differential Equation form}
    \begin{split}
        \Ddot{N}^{(b)}(t)&=\frac{1}{r_{max}K_s N^{(b)}}(K_s r_{max} \dot{N}^{(b)^2}- \gamma \dot{N}^{(b)^3} \\
        & + 2\gamma r_{max} N^{(b)} \dot{N}^{(b)^2} -\gamma r_{max}^2N^{(b)^2}\dot{N}^{(b)} )
    \end{split}
\end{equation}
The existing models though heuristic, suggest that $N^{(b)}$ at any point in time is a function of the past
\begin{align}
    N^{(b)}_{k+1}=f(N^{(b)}_k,N^{(b)}_{k-1},\cdots) \nonumber
\end{align} $N^{(b)}_k$ to be a function of its finite past. This is the general structure of the discrete nonlinear autoregressive (NAR) model given by
\begin{equation}\label{eq: NAR model}
    \begin{split}
        y_{k} &= f(y_{k-1},y_{k-2},\cdots,y_{k-\tau})\\
        y_{i} &\in \mathbb{R}^p \quad \forall i \in \mathbb{Z}_{>0}\\
        f&:\underbrace{\mathbb{R}^p \times \mathbb{R}^p \times \cdots \times \mathbb{R}^p}_{\tau \text{ times}}\rightarrow \mathbb{R}^p
    \end{split}
\end{equation}
where the current output is a function of the past $\tau$ outputs.

\section{Hankel dynamic mode decomposition}\label{sec: Hankel DMD}
Given the nonlinear system (\ref{eq: general nonlinear growth curve model}) with the state measurement given by (\ref{eq: output OD600}) and modeled by the discrete time difference equation (\ref{eq: NAR model}), Hankel DMD \cite{arbabi2017ergodic} is a suitable algorithm to solve the model identification problem with the NAR  structure. The promising feature of using a DMD algorithm is that it identifies a linear state space representation which has a theoretical foundation in Koopman operator theory.

Given the autonomous state space system
\begin{equation}
\begin{split}
    \Tilde{x}_{k+1} &= \Tilde{f}(\Tilde{x}_k)\\
    y_{k} &= h(\Tilde{x}_k)
\end{split}
\end{equation}
where $x_k \in \mathbb{R}^n$ is the state, $\Tilde{f}:\mathbb{R}^n\rightarrow \mathbb{R}^n$ is the dynamics, $y_k \in \mathbb{R}^p$ is the output and $h:\mathbb{R}^n\rightarrow \mathbb{R}^p$ is a nonlinear function that maps the state directly to itself, i.e. $x$ is identical to the output $y$, Hankel DMD constructs a Koopman model of the form
\begin{equation}
    \begin{bmatrix}
    \psi(y_{k+1})  \\ \psi(y_{k+2}) \\ \vdots \\ \psi(y_{k+\tau})
    \end{bmatrix}
    =K\begin{bmatrix}
    \psi(y_{k}) \\ \psi(y_{k+1}) \\ \vdots \\ \psi(y_{k+\tau-1})
    \end{bmatrix}
\end{equation}
such that $\psi:\mathbb{R}^p \rightarrow \mathbb{R}^{N_p}$ is the dictionary of state inclusive observables of the state $\Tilde{x}_k$ constructed by a nonlinear transformation of the corresponding output $y_k$ and $K$ is the Koopman operator.   Regardless of full-state measurements, we nonetheless cast Hankel DMD in this form to compare it with our subsequent causal jump DMD algorithm. 

Given the output measurements $\{y_1,y_2,..,y_{N}\}$, to identify an approximate Koopman operator $K$ using Hankel DMD, the time shifted Hankel matrices are constructed as
\begin{equation}
    \begin{split}
        \Psi(Y_p) &= 
        \begin{bmatrix} 
        \psi(y_{1})&\psi(y_{2})&\hdots & \psi(y_{N-\tau})\\  
        \psi(y_{2})&\psi(y_{3})& \hdots & \psi(y_{N-\tau+1}) \\ 
        \vdots  & \vdots  & \ddots & \vdots \\ 
        \psi(y_{\tau})& \psi(y_{\tau+1})&\hdots & \psi(y_{N-1})   \end{bmatrix} \\
        \Psi(Y_f) &= 
        \begin{bmatrix}
        \psi(y_{2})&\psi(y_{3})&\hdots & \psi(y_{N-\tau+1})\\  
        \psi(y_{3})&\psi(y_{4})& \hdots & \psi(y_{N-\tau+2}) \\ 
        \vdots  & \vdots  & \ddots & \vdots \\ 
        \psi(y_{\tau+1})& \psi(y_{\tau+2})&\hdots & \psi(y_N) 
        \end{bmatrix}
    \end{split}
\end{equation}
and the optimization problem
\begin{equation}\label{eq: Hankel DMD Optimization}
    \min_{K} || \Psi(Z_{f}) - K \Psi(Z_p)|| 
\end{equation}
is solved using the Moore-Penrose pseudoinverse method mentioned in section \ref{sec:Koop}. This yields a solution of the form
\begin{equation}
    \begin{bmatrix}
    \psi(y_{k+1})  \\ \psi(y_{k+2}) \\ \vdots \\ \psi(y_{k+\tau})
    \end{bmatrix}
    =
    \begin{bmatrix}
        0\hspace{-5pt} & \mathbb{I}_{N_p}\hspace{-5pt} & \cdots\hspace{-5pt} & 0 \hspace{-5pt}&0\\ 
        0\hspace{-5pt} & 0\hspace{-5pt} & \cdots\hspace{-5pt} & 0\hspace{-5pt} &0\\ 
        \vdots\hspace{-5pt} & \vdots\hspace{-5pt} & \ddots\hspace{-5pt} &\vdots \hspace{-5pt}& \vdots\\
        0 \hspace{-5pt}& 0\hspace{-5pt} & \cdots \hspace{-5pt}& 0 &\mathbb{I}_{N_p}\\
        k_1\hspace{-5pt} & k_2 \hspace{-5pt} & \cdots \hspace{-5pt} &k_{\tau-1}\hspace{-5pt} & k_\tau\\ 
    \end{bmatrix}
    \begin{bmatrix}
    \psi(y_{k}) \\ \psi(y_{k+1}) \\ \vdots \\ \psi(y_{k+\tau-1})
    \end{bmatrix}\nonumber
\end{equation}
Other than the last $N_p$ equations, the others are trivial. To construct an output predictor, we take the component $y_{k+\tau}$ of $\psi(y_{k+\tau})$ to get
\begin{equation}\label{eq: Hankel DMD predictor}
    y_{k+\tau} = \Tilde{k}_1\psi(y_k) + \Tilde{k}_2\psi(y_{k+1}) +\cdots + \Tilde{k}_\tau\psi(y_{k+\tau +1})
\end{equation}
where $\Tilde{k}_i$ are the components of $k_i$ that map $\psi(y_{k+\tau-1})$ to $y_{k+\tau}$. More generally, this yields a nonlinear equation of the form 
\begin{equation}
    y_k = \Tilde{f}_1(y_{k-1}) + \Tilde{f}_2(y_{k-2}) + \cdots + \Tilde{f}_\tau(y_{k-\tau})
\end{equation}
where the functions $\Tilde{f}_1,\Tilde{f}_2,\cdots,\Tilde{f}_\tau$ have the same basis functions with different coefficients. This identifies a constrained NAR model as it imposes an additive structure on the basis of nonlinear models across time.

\section{Dynamic mode decomposition of nonlinear autoregressive models} \label{sec: NAR-DMD}
To identify a Koopman operator for the unconstrained NAR model (\ref{eq: NAR model}), we formulate a state space representation for the NAR model  with full state observation and identify an approximate Koopman operator for that model using the general class of DMD algorithms like extended DMD and deep DMD.

In this methodology, the problem is broken into two pieces —— the system identification aspect where we select the model structure and the dynamic mode decomposition aspect where we have to construct the dictionary of observables. We define a window parameter $\tau\in\mathbb{Z}_{>0}$ indicating how many past output snapshots are used to define a new extended dictionary of monomial observable functions, up to order $n_o\in\mathbb{Z}_{>0}$.  The new $\tau$-dictionary defines a general extended dynamic mode decomposition problem, which we then solve using classical methods.  

We proceed as follows: given the NAR model (\ref{eq: NAR model}) with the system identification parameter $\tau$, we construct a state defined by 
\begin{equation}\label{eq: 1-step Koopman state}
    z_k:=
    \begin{bmatrix}
        y_{k+1} & 
        y_{k+2} &
        \cdots &
        y_{k+\tau}
    \end{bmatrix}^T
\end{equation}
with $z_k \in \mathbb{R}^{p\tau}$.This yields the state space representation
\begin{align}\label{eq: 1-step Koopman model}
        z_{k+1}&=
        \begin{bmatrix}
            y_{k+2} \\ 
            y_{k+3} \\
            \vdots \\
            y_{k+\tau}\\
            y_{k+\tau+1}
        \end{bmatrix} := 
        \begin{bmatrix}
            f_1(y_{k+1},y_{k+2},\cdots,y_{k+\tau}) \\ 
            f_2(y_{k+1},y_{k+2},\cdots,y_{k+\tau}) \\
            \vdots \\
            f_{\tau-1}(y_{k+1},y_{k+2},\cdots,y_{k+\tau}) \\
            f_\tau(y_{k+1},y_{k+2},\cdots,y_{k+\tau}) \\
        \end{bmatrix}\nonumber\\
        &:=\begin{bmatrix}
            y_{k+2}\\ 
            y_{k+3} \\
            \vdots \\
            y_{k+\tau}\\
            f_(y_{k+1},y_{k+2},\cdots,y_{k+\tau})
        \end{bmatrix} = F(z_k)\nonumber\\
        \Rightarrow z_{k+1} &= \Tilde{F}(z_k)
\end{align}
where $\Tilde{F}:\mathbb{R}^{p\tau}\rightarrow\mathbb{R}^{p\tau}$ represents the dynamics of the lifted "state" model. The approximate EDMD model for the full output observable model is given by
\begin{equation}
    \psi(z_{k+1}) = K\psi(z_k)
\end{equation}
where  $\psi(z_k)$ is the state inclusive dictionary of observables defined as
\begin{equation}
    \psi(z_k) = 
    \begin{bmatrix}
        z_k\\
        \varphi(z_k)
    \end{bmatrix}
\end{equation}
with $\varphi:\mathbb{R}^{p\tau}\rightarrow \mathbb{R}^{N_p}$ being a nonlinear transformation that constructs the nonlinear observables. Since the only additional information in the state $z_{k+1}$ when compared to the state $z_k$ is $y_{k+\tau+1}$, the output predictor form for the Koopman model can be identified considering the complete Koopman model and extracting the equation that corresponds to $y_{k+\tau+1}$ given by
\begin{align}\label{eq: 1-step Koopman predictor}
    \psi(z_{k+1}) &= K \psi(z_{k})\nonumber\\
    \Rightarrow 
    \begin{bmatrix}
        y_{k+2}\\ 
        \vdots \\
        y_{k+\tau} \\
        y_{k+\tau+1}\\
        \varphi(z_{k+1})
    \end{bmatrix} &= 
    \begin{bmatrix}
        \bullet & \cdots & \bullet & \bullet & \bullet\\
        \vdots & \ddots & \vdots &\vdots &\vdots \\
        \bullet & \cdots & \bullet & \bullet & \bullet\\
        k_1 & \cdots & k_{\tau-1} & k_\tau & k_{11}\\
        \bullet & \cdots & \bullet & \bullet & \bullet\\
    \end{bmatrix}
    \begin{bmatrix}
        y_{k+1}\\ 
        \vdots \\
        y_{k+\tau-1} \\
        y_{k+\tau}\\
        \varphi(z_{k})
    \end{bmatrix}\nonumber\\
    \Rightarrow y_{k+\tau+1} &= k_1y_{k+1} + \cdots +  k_\tau y_{k+\tau}\nonumber\\
    &+k^T_{11}\varphi(y_{k+1}, \cdots, y_{k+\tau})
\end{align}
The output predictor form keeps the general structure of the NAR model intact as opposed to the predictor identified by Hankel DMD which identified a constrained model. But, the issue with this model is the causality. It can be seen from (\ref{eq: 1-step Koopman predictor}) that the Koopman model is non causal due to the overlap of outputs $y_k$ between the states $z_{k+1}$ and $z_k$. This identifies models that use future outputs to predict past outputs which are inadmissible as our system is causal. To identify a causal model, the property of (\ref{eq: 1-step Koopman model}) proved in proposition 1 is very important.

\begin{proposition}
Given the state space model (\ref{eq: 1-step Koopman model}) for the nonlinear autoregressive (NAR) model (\ref{eq: NAR model}), if the state is propagated $i$ time steps where $i\in\{1,2,...,\tau\}$
 \begin{equation*}
    \begin{split}
        z_{k+i}&=\Tilde{F}^i(z_k)=\underbrace{\Tilde{F} \circ \Tilde{F} \circ \cdots \circ \Tilde{F}}_{\text{i times}}(z_k),
    \end{split}
 \end{equation*}
 then the last $i$ functions of $\Tilde{F}_H^i(z_k)$ are such that
 \begin{equation}\label{eq: lemma1 statement}
     \begin{split}
        (\Tilde{F}^i)^{(\tau-i+j)}(z_k) &= f^{(j)}(z_k) \quad j\in\{1,2,\hdots,i\}\\
        y_{y+\tau+j} = f^{(j)}(z_k) &= f^{(j)}(y_{k+1},y_{k+2},\hdots,y_{k+\tau}) 
     \end{split}
 \end{equation}
 where $(\Tilde{F}^i)^{(b)}(z_k)$ corresponds to the $b^{th}$ function of $\Tilde{F}^i(z_k)$ and $f^{(j)}(z_k)$ is the $j$-step predictor of the NAR model (\ref{eq: NAR model}).
\end{proposition}
\begin{proof}
    Given the state $z_k$ defined in (\ref{eq: 1-step Koopman state}), the state propagated $i$ time steps $\forall i \in \mathbb{Z}_{\geq0}$ is given by
    \begin{equation*}
    z_{k+i} = 
        \begin{bmatrix} 
            y_{k+i+1} & y_{k+i+2} & \cdots & y_{k+i+\tau} 
        \end{bmatrix}^T
    \end{equation*} 
    and the $m^{th}$ component of $z_{k+i}$ is given by $z_{k+i}^{(m)} =  y_{k+i+m}$ where $m \in \{1,2,...,\tau\}$. 
    
    A function $f^{(j)}: \underbrace{\mathbb{R}^p \times \mathbb{R}^p \times \cdots \times \mathbb{R}^p}_{\tau \text{ times}} \rightarrow \mathbb{R}^{p} $  is a j-step predictor of the NAR model (\ref{eq: NAR model}) if it has the following form
    \begin{equation*}
        \begin{split}
            y_{k+\tau+j} &= f^{(j)}(x_k)=f^{(j)}(y_{k+1},y_{k+2},\cdots,y_{k+\tau}).
        \end{split}
    \end{equation*}

    Now that we have the state definitions and the predictor function definitions in place, we prove (\ref{eq: lemma1 statement}) by induction. For $i=1$, 
    \begin{equation*}
        \begin{split}
            z_{k+1} &= \Tilde{F}^{(1)}(z_{k})\\
            (\Tilde{F}^1)^{(\tau - i +j)} (z_{k})&= (\Tilde{F}^1)^{(\tau)}(z_{k}) = f^{(1)}(x_k)\quad j \in \{1\}\\
            \Rightarrow z_{k+1}^{(\tau)} &= y_{k+\tau+1} = f^{(1)}(z_k)
        \end{split}
    \end{equation*} 
    Hence (\ref{eq: lemma1 statement}) is satisfied for $i=1$. We assume the result is true for $i=p$. This yields
    \begin{equation*}
        \begin{split}
            (\Tilde{F}^p(z_k))^{(\tau-p+j)}(z_{k}) &= f^{(j)}(z_k)  \quad j \in \{1,2,...,p\}\\
            \Rightarrow z_{k+p} = 
            \begin{bmatrix}
                \vspace{-2pt}y_{k+p+1}\\
                \vspace{-2pt}\vdots\\
                y_{k+\tau}\\
                \vspace{-2pt}y_{k+\tau+1}\\
                \vspace{-2pt}\vdots\\
                y_{k+\tau+p}
            \end{bmatrix}
            &=\Tilde{F}^p(z_k) = 
            \begin{bmatrix}
                \vspace{-2pt}y_{k+p+1}\\
                \vspace{-2pt}\vdots\\
                y_{k+\tau}\\
                \vspace{-2pt}f^{(1)}(z_k)\\
                \vspace{-2pt}\vdots\\
                f^{(p)}(z_k)
            \end{bmatrix}.
        \end{split}
    \end{equation*}
    For $i=p+1$, the state $z_{k+p+1}$ becomes
    \begin{equation*}
        \begin{split}
            z_{k+p+1} &=\Tilde{F}^{p+1}(z_k) = \Tilde{F} \circ \Tilde{F}^p(z_k)\\
             \Rightarrow 
             \begin{bmatrix}
                \vspace{-2pt}y_{k+p+2}\hspace{-6pt}\\
                \vspace{-2pt}\vdots\hspace{-6pt}\\
                y_{k+\tau}\hspace{-4pt}\\
                \vspace{-2pt}y_{k+\tau+1}\hspace{-6pt}\\
                \vspace{-2pt}\vdots\hspace{-6pt}\\
                y_{k+\tau+p}\\
                y_{k+\tau+p+1}
            \end{bmatrix}
             &= \Tilde{F}
            \left(
            \begin{bmatrix}
                \vspace{-2pt}y_{k+p+1}\\
                \vspace{-2pt}\vdots\\
                y_{k+\tau}\\
                \vspace{-2pt}f^{(1)}(z_k)\\
                \vspace{-2pt}\vdots\\
                f^{(p-1)}(z_k)\\
                f^{(p)}(z_k)\\
            \end{bmatrix}
            \right)=
            \begin{bmatrix}
                \vspace{-2pt}y_{k+p+2}\\
                \vspace{-2pt}\vdots\\
                y_{k+\tau}\\
                \vspace{-2pt}f^{(1)}(z_k)\\
                \vspace{-2pt}\vdots\\
                f^{(p)}(z_k)\\
                g
            \end{bmatrix}\\
        \end{split}
    \end{equation*}
    where
    \begin{equation*}
    \begin{split}
        g&= f(y_{k+p+1},...,y_{k+\tau},f^{(1)}(z_k),...,f^{(p)}(z_k))\\
        &=f(z_k^{(p+1)},...,z_k^{(\tau)},f^{(1)}(z_k),...,f^{(p)}(z_k))\\
        &:=g(z_k).
    \end{split}
    \end{equation*}
    Since $g$ is a function of only $z_k$ and since $y_{k+\tau+p+1} = g$, $g(z_k)$ satisfies the definition of a predictor function and hence is a $(p+1)$-step predictor of (\ref{eq: NAR model})
    \begin{equation*}
    \begin{split}
        y_{k+\tau+p+1}&=z_{k+p+1}^{(\tau)} = (\Tilde{F}^{p+1}(z_k))^{(\tau)}= f^{(p+1)}(z_k). 
    \end{split}
    \end{equation*}
    Therefore, for $i=p+1$, 
    \begin{equation*}
        (\Tilde{F}^i(z_k))^{(\tau-p-1+j)} = f^{(j)}(z_k) \quad j \in \{1,2,...,(p+1)\}
    \end{equation*}
    stating that the last $(p+1)$ entries of $z_{k+p+1}$ are $f^{(1)}(x)$, $f^{(2)}(x)$, ..., $f^{(p+1)}(x)$. 
    Hence the proof. 
\end{proof} 
 
 To identify a causal Koopman model for the NAR system (\ref{eq: NAR model}), we propagate the model (\ref{eq: 1-step Koopman model}) by $\tau$ time steps to ensure no intersection of outputs between the states $z_{k+1}$ and $z_k$. We define a new state $x_k = z_{k\tau}$ which yields
\begin{align}
    x_k &= z_{k\tau}\nonumber\\
    \Rightarrow x_{k+1} &= z_{k\tau +\tau} = \Tilde{F}^\tau(z_{k\tau}) = F(x_k)\\
    \text{where } F&= \Tilde{F}^\tau  =  \underbrace{\Tilde{F} \circ \Tilde{F} \circ \cdots \Tilde{F}}_{\tau \text{ times}} \nonumber
\end{align}
Using proposition 1, we can say that the nonlinear state space model contains functions that are 1-step, 2-step, ..., $\tau$-step predictors in the following form

\begin{align}\label{eq: tau-jump Koopman model}
        x_{k+1}&=
        \begin{bmatrix}
            y_{k\tau+\tau+1} \\ 
            y_{k\tau +\tau+2} \\
            \vdots \\
            y_{k\tau+2\tau-1}\\
            y_{k\tau+2\tau}
        \end{bmatrix} := 
        \begin{bmatrix}
            f_1(y_{k\tau+1},y_{k\tau+2},\cdots,y_{k\tau+\tau}) \\ 
            f_2(y_{k\tau+1},y_{k\tau+2},\cdots,y_{k\tau+\tau}) \\
            \vdots \\
            f_{\tau-1}(y_{k\tau+1},y_{k\tau+2},\cdots,y_{k\tau+\tau}) \\
            f_\tau(y_{k\tau+1},y_{k\tau+2},\cdots,y_{k\tau+\tau}) \\
        \end{bmatrix}\nonumber\\
        &:=\begin{bmatrix}
            f^{(1)}(y_{k\tau+1},y_{k\tau+2},\cdots,y_{k\tau+\tau}) \\ 
            f^{(2)}(y_{k\tau+1},y_{k\tau+2},\cdots,y_{k\tau+\tau}) \\
            \vdots \\
            f^{(\tau-1)}(y_{k\tau+1},y_{k\tau+2},\cdots,y_{k\tau+\tau}) \\
            f^{(\tau)}(y_{k\tau+1},y_{k\tau+2},\cdots,y_{k\tau+\tau})
        \end{bmatrix} = F(x_k)\nonumber\\
\end{align}
where $f^{(i)}$ is the $i$-step predictor of the NAR model. We prove the existence of a Koopman operator for this model in Proposition 2.

\begin{proposition}
If the function $f(x)$ in the NAR model (\ref{eq: NAR model}) is analytic, then a Koopman operator exists for (\ref{eq: 1-step Koopman model}) and (\ref{eq: tau-jump Koopman model}).
\end{proposition}

\begin{proof}
Since $f$ in (\ref{eq: NAR model}) is analytic, $\Tilde{F}$ in (\ref{eq: 1-step Koopman model}) is analytic since all the entries of $\Tilde{F}$ are either linear functions or are equal to $f$.
Since $F$ is obtained by the composition of $\Tilde{F}$ $\tau$ times, $F$ is also analytic. 

$F(x)$ admits a countable-dimension Koopman operator $K_x$, with an invariant subspace isomorphic to either a finite or an infinite Taylor polynomial basis \cite{yeung2018koopman}.  Moreover, isomorphism with a Taylor polynomial basis ensures that the Koopman observable space contains the full state observable, i.e. it is state inclusive. 

There are two easy arguments to conclude the proof.  First, note that since $f$ is analytic, $f^\tau$ is analytic and thus by the same reasoning as in \cite{yeung2018koopman}, $f^\tau$ thus must admit a Koopman operator.  The second argument is a constructive one, noting that equation
\begin{equation}
\psi(x[(k)\tau ]) = K^\tau \psi(x[(k-1)(\tau))
\end{equation}
must hold due to $\tau$ applications of the 1-step Koopman equation.  This means therefore that the following {\it matrix} equation must hold
\begin{equation}
\psi\left(\begin{bmatrix}  x[(k)\tau ] \\ x[k\tau+1] \\ \vdots \\ x[(k+1)\tau-1]    
\end{bmatrix}\right)= {\bf K}_J\psi\left(\begin{bmatrix}  (x[(k-1)(\tau)) ] \\ (x[(k-1)]\tau +1]) \\ \vdots\\ (x[(k)\tau-1)]    
\end{bmatrix}\right)
\end{equation}
where ${\bf K}_J = \text{diag}\left( K^\tau , K^\tau , \hdots  K^\tau \right).$ This concludes the proof.
\end{proof}

Since the existence of a Koopman operator has been proved for the model (\ref{eq: tau-jump Koopman model}) in Proposition 2, we construct a state inclusive dictionary of observables
\begin{equation}
    \psi(x_k) = 
    \begin{bmatrix}
        x_k\\
        \varphi(x_k)
    \end{bmatrix}
\end{equation}
with $\varphi:\mathbb{R}^{p\tau}\rightarrow\mathbb{R}^{N_p}$ to define a Koopman model
\begin{equation}
    \psi(x_{k+1}) = K\psi(x_k)
\end{equation}
This Koopman model is causal since there is no intersection of outputs between $x_{k+1}$ and $x_k$. The added feature of this model is that the DMD algorithm while identifying a Koopman operator, also simultaneously minimizes the 1-step, 2-step, ..., $\tau$-step prediction error of the NAR model. 

Now that we have a theoretical state space representation of a NAR model and established the conditions under which a Koopman operator exists, we turn our attention to the algorithm for identification of the Koopman operator. Given the data with M data sets and N data points in each data set $\{y^{(i)}_1,y^{(i)}_2,...,y^{(i)}_N\}$ where $i \in \{1,2,...M\}$ is the index of the data set, we construct the Hankel states $z_k$ and the dictionary of observables allowing the intermixing of states. We compile the observables into snapshot matrices $\Tilde{\Psi}_f(z)$  and $\Tilde{\Psi}_p(z)$ with a $\tau$ time step jump and solve the Koopman learning problem
\begin{equation*}
    ||\Tilde{\Psi}_f(z)-K\Tilde{\Psi}_p(z)||_F
\end{equation*}
using the methodology in Algorithm 1.
 
\begin{algorithm}
    \label{Algo: CJ DMD}
	\caption{Extended DMD for NAR models} 
	\begin{algorithmic}[1]
	    \State Get NAR model parameter $\tau$
	    \State Get extended DMD parameter $n_o$ for monomial observables
		\For {dataset $i =1,2,\ldots,M$}
			\For {time index $j=1,2,\ldots,N-\tau$}
			    \State Construct the Hankel state
			    \begin{equation*}
			        z^{(i)}_j = \begin{bmatrix}y^{(i)}_{j+1} & y^{(i)}_{j+2} & \cdots y^{(i)}_{j+\tau}\end{bmatrix}
			    \end{equation*}
				\State Construct the dictionary of observables
				$\psi(z^{(i)}_j)$
			\EndFor
			\State Construct the snapshot matrices for each data set with the $\tau$-jump
			\begin{equation*}
			    \begin{split}
			        \Psi^{(i)}_p(x)&=
                \begin{bmatrix}
                    \psi(z^{(i)}_1)&\psi(z^{(i)}_2)&...&\psi(z^{(i)}_{N-2\tau})
                \end{bmatrix}\\
                \Psi^{(i)}_f(x)&=
                \begin{bmatrix}
                    \psi(z^{(i)}_{1+\tau})&\psi(z^{(i)}_{2+\tau})&...&\psi(z^{(i)}_{N-\tau})
                \end{bmatrix}\\
			    \end{split}
			\end{equation*}
		\EndFor
		\State Compile the snapshot matrices across data sets
		\begin{equation*}
		        \begin{split}
                    \Tilde{\Psi}_p(x)=
                    \begin{bmatrix}
                        \Psi^{(1)}_p(x)&\Psi_p^{(2)}(x)&...&\Psi_p^{(M)}(x)
                    \end{bmatrix}\nonumber\\
                    \Tilde{\Psi}_f(x)=
                    \begin{bmatrix}
                        \Psi_f^{(1)}(x)&\Psi_f^{(2)}(x)&...&\Psi_f^{(M)}(x)
                    \end{bmatrix}\
                \end{split}
		\end{equation*}
		\State Compute the SVD of $\Tilde{\Psi}_p(x)= USV^*$
        \State Truncate to required number of singular values and identify the Koopman operator
        \begin{equation*}
            \hat{K} = \Tilde{\Psi}_f(x)\Tilde{V}\Tilde{S}^{-1}\Tilde{U}^*
        \end{equation*}
	\end{algorithmic} 
\end{algorithm}

\section{Results}\label{sec: Results}
From the data-sets obtained in the plate reader experiments shown in Fig. \ref{fig:96well}, we used Algorithm 1 to implement extended DMD using monomials as the dictionary of observables
\begin{multline}
\psi(z_k)=
    [
        y_{k+1},...,y_{k+\tau},y^2_{k+1},y_{k+1}y_{k+2},...,y^2_{k+\tau}\\y^3_{k+1},y^2_{k+1}y_{k+2},...] ^T . \nonumber\ 
\end{multline}
to identify an approximate Koopman operator for the state space model (\ref{eq: tau-jump Koopman model}) as a solution to the identification of the NAR model (\ref{eq: NAR model}). 

We use all the datasets in Fig. \ref{fig:96well} to find a Koopman operator invariant to the substrate concentrations. They are broken equally into training, validation and test set. Given the two parameters $\tau$ (NAR model parameter) and $n_o$ (extended DMD parameter), we can find the optimal approximate Koopman operator by cumulatively iterating through the principal components and evaluating the summation of the mean squared error(MSE) of training and validation data. The number of principal components corresponding to the minimum MSE yields the optimal Koopman operator for a given $\tau$ and $n_o$.  We then iterate through the two parameters to find the optimal model that minimizes the

 By choosing $\tau=9$ and keeping the maximum order of monomials to 3, the Koopman operator has been identified and the prediction on the training data is shown in Figure \ref{fig:Dyntrain} and it has an MSE of 3.4\%. The identified Koopman operator has an MSE of 9\% and the fit is shown in Figure \ref{fig:Dyntest}. 
\begin{figure}[h]
  \centering
  \includegraphics[width=1\columnwidth]{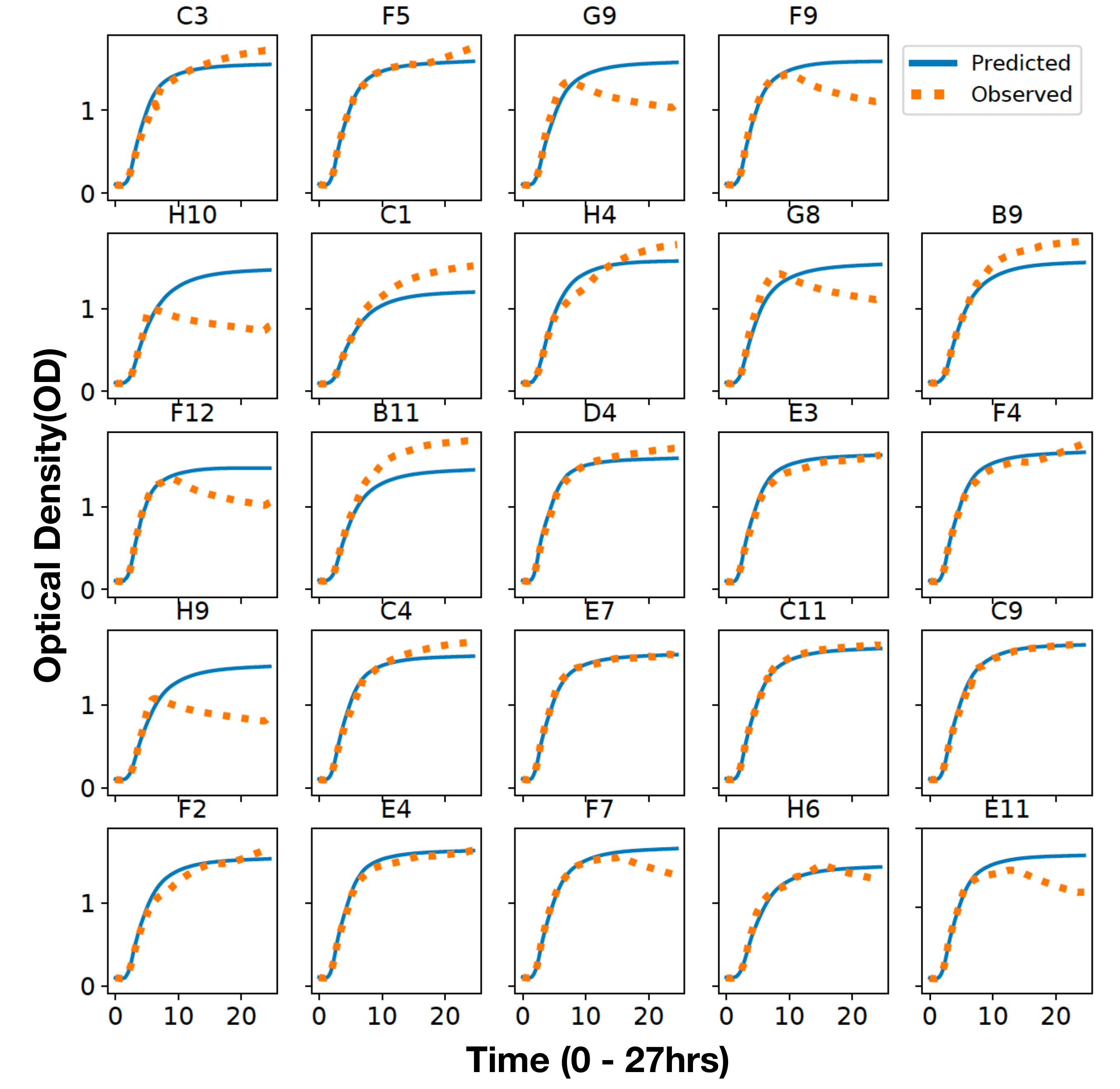}
    \caption{The identified Koopman operator is tested on the training sets with 9 point initial condition and up to $3^{rd}$ order monomials to get a MSE of 3.4\%}
    \label{fig:Dyntrain}
\end{figure}
\begin{figure}[h]
  \centering
  \includegraphics[width=1\columnwidth]{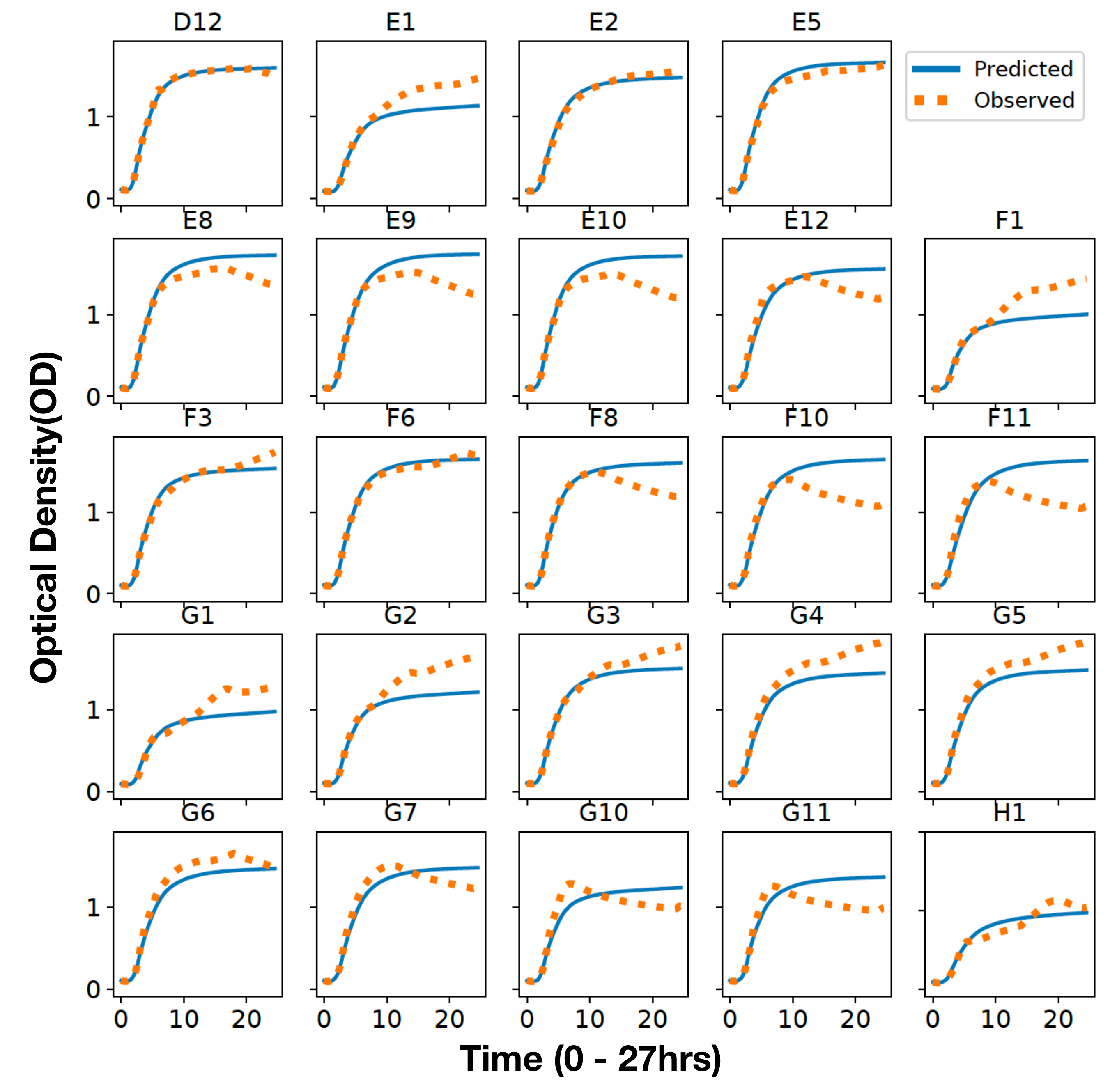}
    \caption{The identified Koopman operator is tested on the test sets by using the initial observables$\psi(x_0)$ and the mean squared error remains the same as that of the training set.}
    \label{fig:Dyntest}
\end{figure}

The results on the experimental data suggest that Causal Jump DMD is a suitable candidate algorithm for identifying the Koopman operator of the population growth dynamics of bacteria and can also be extended in general to identify Koopman operators for NAR models.

\section{Conclusion}\label {conclusion}
In this paper, we introduced the microbial growth curve dynamics to motivate the usage of DMD algorithms to identify Koopman operators for NAR models. We formulated Hankel DMD as a state space representations of the NAR model and showed that it is restrictive in its structure. We construct a causal state space model for the NAR model and identify a Koopman operator for it using extended dynamic mode decomposition with a monomial dictionary of observables. We showed that it does a good job in predicting the population growth dynamics of \textit{Pseudomonas putida} invariant to substrate concentrations. The future goals of this work is to use this model to identify the optimal media conditions for maximal and minimal growth of the microbe thereby enabling us to develop a general methodology to develop an external growth harness for microbes for dynamic growth control. To achieve this, we need to extend the mathematical models to allow for inputs and extend the identification to NARX and NARMAX models. Further, if we integrate this framework with deepDMD which aids in finding the observable functions in a parsimonious fashion, it renders a useful tool for identifying high dimensional linear models for nonlinear systems.

\section{Acknowledgments}
The authors gratefully acknowledge the funding of DARPA grants FA8750-17-C-0229, HR001117C0092, HR001117C0094, DEAC0576RL01830. The authors would also like to thank Professors Arun K. Tangirala, Igor Mezic, Milan Korda, Alexandre Mauroy, Nathan Kutz, Steve Haase, John Harer, Devin Strickland, and Eric Klavins for insightful discussions.  
Any opinions, findings, conclusions, or recommendations expressed in this material are those of the authors and do not necessarily reflect the views of the Defense Advanced Research Project Agency, the Department of Defense, or the United States government. This material is based on work supported by DARPA and AFRL under contract numbers FA8750-17-C-0229, HR001117C0092, HR001117C0094, DEAC0576RL01830.

\bibliographystyle{IEEEtran}
\bibliography{references}

\begin{thebibliography}{10}
\providecommand{\url}[1]{#1}
\csname url@samestyle\endcsname
\providecommand{\newblock}{\relax}
\providecommand{\bibinfo}[2]{#2}
\providecommand{\BIBentrySTDinterwordspacing}{\spaceskip=0pt\relax}
\providecommand{\BIBentryALTinterwordstretchfactor}{4}
\providecommand{\BIBentryALTinterwordspacing}{\spaceskip=\fontdimen2\font plus
\BIBentryALTinterwordstretchfactor\fontdimen3\font minus
  \fontdimen4\font\relax}
\providecommand{\BIBforeignlanguage}[2]{{%
\expandafter\ifx\csname l@#1\endcsname\relax
\typeout{** WARNING: IEEEtran.bst: No hyphenation pattern has been}%
\typeout{** loaded for the language `#1'. Using the pattern for}%
\typeout{** the default language instead.}%
\else
\language=\csname l@#1\endcsname
\fi
#2}}
\providecommand{\BIBdecl}{\relax}
\BIBdecl

\bibitem{kornberg1980replication}
D.~Kornberg and D.~TA, ``Replication,'' \emph{San Francisco: W H. Freeman},
  1980.

\bibitem{mathon2001cell}
N.~F. Mathon and A.~C. Lloyd, ``Cell senescence and cancer,'' \emph{Nature
  Reviews Cancer}, vol.~1, no.~3, p. 203, 2001.

\bibitem{wu2016metabolic}
G.~Wu, Q.~Yan, J.~A. Jones, Y.~J. Tang, S.~S. Fong, and M.~A. Koffas,
  ``Metabolic burden: cornerstones in synthetic biology and metabolic
  engineering applications,'' \emph{Trends in biotechnology}, vol.~34, no.~8,
  pp. 652--664, 2016.

\bibitem{glazier2015metabolic}
D.~S. Glazier, ``Is metabolic rate a universal ‘pacemaker’for biological
  processes?'' \emph{Biological Reviews}, vol.~90, no.~2, pp. 377--407, 2015.

\bibitem{de2016growth}
D.~De~Martino, F.~Capuani, and A.~De~Martino, ``Growth against entropy in
  bacterial metabolism: the phenotypic trade-off behind empirical growth rate
  distributions in e. coli,'' \emph{Physical biology}, vol.~13, no.~3, p.
  036005, 2016.

\bibitem{sanchez2017improving}
B.~J. Sanchez, C.~Zhang, A.~Nilsson, P.-J. Lahtvee, E.~J. Kerkhoven, and
  J.~Nielsen, ``Improving the phenotype predictions of a yeast genome-scale
  metabolic model by incorporating enzymatic constraints,'' \emph{Molecular
  systems biology}, vol.~13, no.~8, 2017.

\bibitem{zwietering1990modeling}
M.~Zwietering, I.~Jongenburger, F.~Rombouts, and K.~Van't~Riet, ``Modeling of
  the bacterial growth curve,'' \emph{Appl. Environ. Microbiol.}, vol.~56,
  no.~6, pp. 1875--1881, 1990.

\bibitem{tschirhart2019synthetic}
T.~Tschirhart, V.~Shukla, E.~E. Kelly, Z.~Schultzhaus, E.~NewRingeisen, J.~S.
  Erickson, Z.~Wang, W.~garcia, E.~Curl, R.~G. Egbert \emph{et~al.},
  ``Synthetic biology tools for the fast-growing marine bacterium vibrio
  natriegens,'' \emph{ACS synthetic biology}, 2019.

\bibitem{Khan2018broad}
N.~Khan, E.~Yeung, Y.~Farris, S.~J. Fansler, and H.~C. Bernstein, ``A
  broad-host-range event detector: expanding and quantifying performance across
  bacterial species,'' \emph{bioRxiv}, p. 369967, 2018.

\bibitem{gill1979effect}
C.~Gill and K.~Tan, ``Effect of carbon dioxide on growth of pseudomonas
  fluorescens.'' \emph{Appl. Environ. Microbiol.}, vol.~38, no.~2, pp.
  237--240, 1979.

\bibitem{gulliver2016comparative}
D.~M. Gulliver, G.~V. Lowry, and K.~B. Gregory, ``Comparative study of effects
  of co2 concentration and ph on microbial communities from a saline aquifer, a
  depleted oil reservoir, and a freshwater aquifer,'' \emph{Environmental
  Engineering Science}, vol.~33, no.~10, pp. 806--816, 2016.

\bibitem{LaBauve2012growth}
A.~E. LaBauve and M.~J. Wargo, ``Growth and laboratory maintenance of
  pseudomonas aeruginosa,'' \emph{Current protocols in microbiology}, vol.~25,
  no.~1, pp. 6E--1, 2012.

\bibitem{Palacios2014bayesian}
A.~P. Palacios, J.~M. Mar{\'\i}n, E.~J. Quinto, M.~P. Wiper \emph{et~al.},
  ``Bayesian modeling of bacterial growth for multiple populations,'' \emph{The
  Annals of Applied Statistics}, vol.~8, no.~3, pp. 1516--1537, 2014.

\bibitem{lee2016vibrio}
H.~H. Lee, N.~Ostrov, B.~G. Wong, M.~A. Gold, A.~Khalil, and G.~M. Church,
  ``Vibrio natriegens, a new genomic powerhouse,'' \emph{bioRxiv}, p. 058487,
  2016.

\bibitem{mezic2005spectral}
I.~Mezic, ``Spectral properties of dynamical systems, model reduction and
  decompositions,'' \emph{Nonlinear Dynamics}, vol.~41, no. 1-3, pp. 309--325,
  2005.

\bibitem{budivsic2012applied}
M.~Budi{\v{s}}i{\'c}, R.~Mohr, and I.~Mezi{\'c}, ``Applied koopmanism,''
  \emph{Chaos: An Interdisciplinary Journal of Nonlinear Science}, vol.~22,
  no.~4, p. 047510, 2012.

\bibitem{williams2015data}
M.~O. Williams, I.~G. Kevrekidis, and C.~W. Rowley, ``A data--driven
  approximation of the koopman operator: Extending dynamic mode
  decomposition,'' \emph{Journal of Nonlinear Science}, vol.~25, no.~6, pp.
  1307--1346, 2015.

\bibitem{rowley2009spectral}
C.~W. Rowley, I.~Mezi{\'c}, S.~Bagheri, P.~Schlatter, and D.~S. Henningson,
  ``Spectral analysis of nonlinear flows,'' \emph{Journal of fluid mechanics},
  vol. 641, pp. 115--127, 2009.

\bibitem{proctor2016dynamic}
J.~L. Proctor, S.~L. Brunton, and J.~N. Kutz, ``Dynamic mode decomposition with
  control,'' \emph{SIAM Journal on Applied Dynamical Systems}, vol.~15, no.~1,
  pp. 142--161, 2016.

\bibitem{williams2016extending}
M.~O. Williams, M.~S. Hemati, S.~T. Dawson, I.~G. Kevrekidis, and C.~W. Rowley,
  ``Extending data-driven koopman analysis to actuated systems,''
  \emph{IFAC-PapersOnLine}, vol.~49, no.~18, pp. 704--709, 2016.

\bibitem{askham2018variable}
T.~Askham and J.~N. Kutz, ``Variable projection methods for an optimized
  dynamic mode decomposition,'' \emph{SIAM Journal on Applied Dynamical
  Systems}, vol.~17, no.~1, pp. 380--416, 2018.

\bibitem{kaneko2019convolutional}
Y.~Kaneko, S.~Muramatsu, H.~Yasuda, K.~Hayasaka, Y.~Otake, S.~Ono, and
  M.~Yukawa, ``Convolutional-sparse-coded dynamic mode decomposition and its
  application to river state estimation,'' in \emph{ICASSP 2019-2019 IEEE
  International Conference on Acoustics, Speech and Signal Processing
  (ICASSP)}.\hskip 1em plus 0.5em minus 0.4em\relax IEEE, 2019, pp. 1872--1876.

\bibitem{azencot2019consistent}
O.~Azencot, W.~Yin, and A.~Bertozzi, ``Consistent dynamic mode decomposition,''
  \emph{arXiv preprint arXiv:1905.09736}, 2019.

\bibitem{manohar2019optimized}
K.~Manohar, E.~Kaiser, S.~L. Brunton, and J.~N. Kutz, ``Optimized sampling for
  multiscale dynamics,'' \emph{Multiscale Modeling \& Simulation}, vol.~17,
  no.~1, pp. 117--136, 2019.

\bibitem{schmid2010dynamic}
P.~J. Schmid, ``Dynamic mode decomposition of numerical and experimental
  data,'' \emph{Journal of fluid mechanics}, vol. 656, pp. 5--28, 2010.

\bibitem{sinha2019computation}
S.~Sinha and E.~Yeung, ``On computation of koopman operator from sparse data,''
  \emph{arXiv:1901.03024}, 2019.

\bibitem{yeung2019learning}
E.~Yeung, S.~Kundu, and N.~Hodas, ``Learning deep neural network
  representations for koopman operators of nonlinear dynamical systems,'' in
  \emph{2019 American Control Conference (ACC)}.\hskip 1em plus 0.5em minus
  0.4em\relax IEEE, 2019, pp. 4832--4839.

\bibitem{hasnain2019data}
A.~Hasnain, S.~Sinha, Y.~Dorfan, A.~E. Borujeni, Y.~Park, P.~Maschhoff,
  U.~Saxena, J.~Urrutia, N.~Gaffney, D.~Becker \emph{et~al.}, ``A data-driven
  method for quantifying the impact of a genetic circuit on its host,''
  \emph{arXiv preprint arXiv:1909.06455}, 2019.

\bibitem{Johnson2018class}
C.~A. Johnson and E.~Yeung, ``A class of logistic functions for approximating
  state-inclusive koopman operators,'' in \emph{2018 Annual American Control
  Conference (ACC)}.\hskip 1em plus 0.5em minus 0.4em\relax IEEE, 2018, pp.
  4803--4810.

\bibitem{otto2019linearly}
S.~E. Otto and C.~W. Rowley, ``Linearly recurrent autoencoder networks for
  learning dynamics,'' \emph{SIAM Journal on Applied Dynamical Systems},
  vol.~18, no.~1, pp. 558--593, 2019.

\bibitem{takeishi2017learning}
N.~Takeishi, Y.~Kawahara, and T.~Yairi, ``Learning koopman invariant subspaces
  for dynamic mode decomposition,'' in \emph{Advances in Neural Information
  Processing Systems}, 2017, pp. 1130--1140.

\bibitem{li2017extended}
Q.~Li, F.~Dietrich, E.~M. Bollt, and I.~G. Kevrekidis, ``Extended dynamic mode
  decomposition with dictionary learning: A data-driven adaptive spectral
  decomposition of the koopman operator,'' \emph{Chaos: An Interdisciplinary
  Journal of Nonlinear Science}, vol.~27, no.~10, p. 103111, 2017.

\bibitem{you2018deep}
P.~You, J.~Pang, and E.~Yeung, ``Deep koopman controller synthesis for
  cyber-resilient market-based frequency regulation,''
  \emph{IFAC-PapersOnLine}, vol.~51, no.~28, pp. 720--725, 2018.

\bibitem{yeung2018koopman}
E.~Yeung, Z.~Liu, and N.~O. Hodas, ``A koopman operator approach for computing
  and balancing gramians for discrete time nonlinear systems,'' in \emph{2018
  Annual American Control Conference (ACC)}.\hskip 1em plus 0.5em minus
  0.4em\relax IEEE, 2018, pp. 337--344.

\bibitem{monod1949growth}
J.~Monod, ``The growth of bacterial cultures,'' \emph{Annual review of
  microbiology}, vol.~3, no.~1, pp. 371--394, 1949.

\bibitem{brandt2003general}
B.~W. Brandt, I.~M. van Leeuwen, and S.~A. Kooijman, ``A general model for
  multiple substrate biodegradation. application to co-metabolism of
  structurally non-analogous compounds,'' \emph{Water research}, vol.~37,
  no.~20, pp. 4843--4854, 2003.

\bibitem{kompala1986investigation}
D.~S. Kompala, D.~Ramkrishna, N.~B. Jansen, and G.~T. Tsao, ``Investigation of
  bacterial growth on mixed substrates: experimental evaluation of cybernetic
  models,'' \emph{Biotechnology and Bioengineering}, vol.~28, no.~7, pp.
  1044--1055, 1986.

\bibitem{arbabi2017ergodic}
H.~Arbabi and I.~Mezic, ``Ergodic theory, dynamic mode decomposition, and
  computation of spectral properties of the koopman operator,'' \emph{SIAM
  Journal on Applied Dynamical Systems}, vol.~16, no.~4, pp. 2096--2126, 2017.

\end{thebibliography}

\end{document}